\documentclass[12pt]{amsart}
\usepackage{amsfonts, amssymb, amsmath, amscd, graphics} 
\usepackage[all]{xy}

\newcommand{\C}{\mathbb{C}}
\newcommand{\tr}[1]{\mathrm{tr}(#1)}
\newcommand{\xb}{\mathbf{X}}
\newcommand{\yb}{\mathbf{Y}}
\newcommand{\zb}{\mathbf{Z}}

\newcommand{\ub}{\mathbf{U}}
\newcommand{\vb}{\mathbf{V}}
\newcommand{\wb}{\mathbf{W}}
\newcommand{\gt}{\mathtt{g}}

\newcommand{\aq}{/\!\!/}
\newcommand{\X}{\mathfrak{X}}
\newcommand{\Y}{\mathfrak{Y}}
\newcommand{\R}{\mathfrak{R}}

\newcommand{\SL}{\mathrm{SL}(3,\C)}

\newcommand{\F}{\mathtt{F}}
\newcommand{\G}{\mathfrak{G}}

\newcommand{\id}{\mathbf{I}}

\newcommand{\SLm}[1]{\mathrm{SL}(#1,\C)}
\newcommand{\GLm}[1]{\mathrm{GL}(#1,\C)}
\newcommand{\glm}[1]{\mathfrak{gl}(#1,\C)}

\newtheorem{thm}{Theorem}
\newtheorem{defn}[thm]{Definition}

\newtheorem{cor}[thm]{Corollary}
\newtheorem{rem}[thm]{Remark}
\newtheorem{prop}[thm]{Proposition}

\newtheorem{conj}[thm]{Conjecture}

\title[Algebraic Independence in Character Varieties]{Algebraic Independence in $\SL$ Character Varieties of Free Groups}
\author[S. Lawton]{Sean Lawton}
\address{Mathematics Department, University of Texas-Pan American}
\address{E-mail address: $\mathtt{lawtonsd@utpa.edu}$}
\address{URL: $\mathtt{http://www.math.utpa.edu/ \sim lawtonsd}$}
\date{\today}
\keywords{character variety, free group, algebraic independence}

\begin{document}

\bibliographystyle{amsalpha}

\maketitle

\begin{abstract}
Let $\X_r$ be the moduli space of $\SL$ representations of a free group of rank $r$.  In this paper we describe maximal algebraically independent 
subsets of certain minimal sets of coordinate functions on $\X_r$.  These subsets locally parametrize the moduli space.
\end{abstract}


\section{Introduction}
It is the purpose of this paper to describe local parameters (maximal independent coordinates) for the moduli space of $\SL$-representations of 
free groups of arbitrary rank.  We begin by reviewing a recent description of global coordinates of this moduli space (universal parameter space), 
and then apply the method of \cite{ATZ} to prove that certain subsets of these coordinate functions admit no (polynomial) relations.  
This establishes that these subsets are algebraically independent.  They are maximal with this property and at 
generic non-singular points provide local parameters that determine linearly independent tangent vectors.

Let $\F_r=\langle \gt_1,...,\gt_r\rangle$ be a rank $r$ free group (non-abelian).  Any homomorphism $\rho:\F_r\to \SLm{m}\subset\GLm{m}$ is a 
complex $m$ dimensional representation of $\F_r$.  We call the set $\R_r=\mathrm{Hom}(\F_r,\SLm{m})$ the $\SLm{m}$-\emph{representation variety of} 
$\F_r$.  The evaluation map, $$\rho\mapsto (\rho(\gt_1),...,\rho(\gt_r)),$$ gives a bijection between $\R_r$ and $\SLm{m}^{\times r}$.  
Since $\SLm{m}$ is a smooth affine variety (irreducible, non-singular, algebraic set), $\R_r$ is likewise a smooth affine variety.

Let $\G=\SLm{m}$.  The conjugation action of $\G$ on $\R_r$ is rational; that is, $\G\times \R_r\to \R_r$ is regular, or the mapping is by polynomials 
in the matrix entries of $\G$ in the coordinates of $\R_r$.  In particular, this action is either $(g,\rho)\mapsto g\rho g^{-1}$ or 
$$\left(g,\left(\rho(\gt_1),...,\rho(\gt_r)\right)\right)\mapsto \left(g\rho(\gt_1)g^{-1},...,g\rho(\gt_r)g^{-1}\right)$$ depending on whether we are 
working with $\mathrm{Hom}(\F_r,\G)$ or $\G^{\times r}$, respectively.  We often switch back and forth as is convenient.

Let $K$ be a compact Lie group.  Then $K$ is a real algebraic group which embeds in $\mathrm{O}(n,\mathbb{R})$ for some $n$.  
Since $K$ is algebraic there is an ideal $\mathfrak{I}$ in the real coordinate ring $\mathbb{R}[\mathrm{O}(n,\mathbb{R})]$ defining its points.  
Let $G$ be the complex zeros of $\mathfrak{I}$.  Then $G$ is a complex algebraic subgroup of $\mathrm{O}(n,\C)$ (denoted by $K_\C$ and called 
the \emph{complexification} of $K$) with coordinate ring $\C[G]=\mathbb{R}[K]\otimes_{\mathbb{R}}\C$.  Any complex algebraic group $G$
which arises in this fashion is called reductive.  The ``unitary trick'' shows $\SLm{m}$ is reductive.  We note that this definition, although not 
the most general, coincides with all more general notions of reductivity when the algebraic group is complex.  In particular, another 
equivalent definition is that a complex algebraic group $G$ is reductive if for every finite dimensional representation of $G$ all 
subrepresentations have invariant complements.  The important observation is that such groups act like and have the algebraic structure of 
compact groups.  See \cite{Sch2}.

A theorem of Nagata \cite{Na} says that if a reductive group acts on a finitely generated domain $A$, the the subdomain of invariants 
$A^G=\{a\in A\ |\ g\cdot a=a\}$ is likewise finitely generated.  This is one answer to Hilbert's fourteeth problem.

Since $\R_r$ is an affine variety, its coordinate ring $\C[\R_r]$ is a finitely generated domain, and since $\G$ acts on $\R_r$ it also acts 
on $\C[\R_r]$ by $(g,f(\rho))\mapsto f(g^{-1}\rho g)$.  Since $\mathrm{SU}(m)_\C=\SLm{m}$, $\G$ is reductive.  Thus $\C[\R_r]^{\G}$ is a finitely 
generated domain, and consequently we define $$\X_r=\mathrm{Spec}_{max}\left(\C[\R_r]^{\G}\right),$$ the set of maximal ideals, to be the 
$\G$-\emph{character variety of} $\F_r$.

It can be shown that $\X_r$ is the categorical quotient $\R_r\aq\G$ in the category of affine varieties (or Hausdorff spaces or complex analytic 
varieties \cite{Lu2,Lu3}).  We recall the definition of a categorical quotient to be concrete.

\begin{defn}
A {\it categorical quotient} of a variety $Y$ with an algebraic group $G$ acting rationally is an object $Y\aq G$ and a $G$-invariant morphism 
$\pi:Y\to Y\aq G$ such that the following commutative diagram exists 
uniquely for all invariant morphisms $f:Y\to Z$: $$ \xymatrix{
Y \ar[rr]^-{\pi} \ar[dr]_{f} & & Y\aq G \ar@{.>}[dl]\\
& Z &} $$
It is a {\it good} categorical quotient if the following conditions also hold:
\begin{enumerate}
\item[$(i)$] for open subsets $U\subset Y\aq G$, $\C[U]\approx \C[\pi^{-1}(U)]^{G}$;
\item[$(ii)$] $\pi$ maps closed invariant sets to closed sets;
\item[$(iii)$] $\pi$ separates closed invariant sets.
\end{enumerate}
\end{defn}

When $G$ is reductive and $Y$ is an affine variety, then $Y\to \mathrm{Spec}_{max}(\C[Y]^G)$ is a good categorical quotient.  
See \cite{Do} for details.

It is worth giving a more down to earth description of this quotient.  There is a one-to-one correspondence between the points of $\X_r$ and the 
orbits of {\it completely reducible} representations (representations that are sums of irreducibles); these are the points whose orbits are closed.  
Any representation can be continuously and conjugate-invariantly deformed to one that is completely reducible, so the points of $\X_r$ are unions 
of orbits of representations that are deformable in this way.  Such a union is called an {\it extended orbit equivalence class}.  The character 
variety $\X_r$ may be accurately thought of as either the usual orbit space of $\R_r$ with the non-completely reducible representations removed, 
or as the the usual orbit space with extended orbit equivalences.  Either way, the resulting space is, or is in one-to-one correspondence with, 
an affine algebraic set, irreducible and singular, that satisfies the diagrammatic requirements needed to be a categorical quotient.

Any such reductive quotient has an affine lift (see \cite{MFK}).  In otherwords, there is an affine space $\mathbb{A}^N$ for some potentially 
large $N$ where $\R_r\subset \mathbb{A}^N$ and where the action of $\G$ extends.  Then
$$\Pi:\C[\mathbb{A}^N]\longrightarrow \C[\R_r]$$ and more importantly
$$\Pi_{\G}:\C[\mathbb{A}^N\aq\G]\longrightarrow \C[\R_r\aq \G]$$ are surjective morphisms.  We may take $\mathbb{A}^N$ in our case to 
be $\mathfrak{gl}(m,\C)^{\times r}$ and the action of $\G$ to be, as it is on $\G^{\times r}$, diagonal conjugation.

The coordinate ring of this affine space is $$\C[\mathfrak{gl}(m,\C)^{\times r}]=\C[x^{k}_{ij}\ |\ 1\leq i,j \leq m, \ 1\leq k\leq r],$$ the 
complex polynomial ring in $rm^2$ variables.

Let $$\xb_k=\left(
\begin{array}{cccc}
x^k_{11} & x^k_{12} & \cdots & x^k_{1m}\\
x^k_{21} & x^k_{22}  &\cdots & x^k_{2m}\\
\vdots &\vdots & \ddots & \vdots\\
x^k_{m1} & x^k_{m2} &\cdots &x^k_{mm}\\
\end{array}\right)$$  be a \emph{generic matrix} of size $m\times m$.

In 1976 Procesi \cite{P1} proves

\begin{thm}[$1^{\text{st}}$ Fundamental Theorem of Invariants of $m\times m$ Matrices]\label{prfd1}
$$\C[\glm{m}^{\times r}\aq \G]=\C[\tr{\xb_{i_1}\xb_{i_2}\cdots \xb_{i_l}}\ |\ 1\leq l \leq d(m)],$$ where $d(m)$ is a fixed positive integer 
dependent only on $m$. 
\end{thm}

However, according to \cite{F} the above result had been discovered earlier in \cite{Ki}.

The number $d(m)$ is called the \emph{degree of nilpotency}.  The only values known are $d(2)=3$, $d(3)=6$ and $d(4)=10$, but it is known that 
$d(m)\leq m^2$ (see \cite{Ra}).  Therefore, 
$\C[\R_m\aq \G]$ is generated by $\Pi(\tr{\xb_{i_1}\xb_{i_2}\cdots \xb_{i_l}})=\tr{\widehat{\xb_{i_1}}\widehat{\xb_{i_2}}\cdots \widehat{\xb_{i_l}}}$ 
where $\widehat{\xb_k}=(\widehat{x^k_{ij}})$ and $\widehat{x_{ij}^k}=\Pi(x^k_{ij})\in \C[\R_m].$

In \cite{L4} minimal generators for $\C[\SLm{3}^{\times r}\aq \SLm{3}]$ are constructed, providing global coordinates for its spectrum of maximal 
ideals.  The main theorem from \cite{L4} is 

\begin{thm}\label{maintheorem}
$\C[\SL^{\times r}\aq \SL]$ is minimally generated by 
$\binom{r}{1}$ invariants of the form $\tr{\xb}$, 
$\binom{r}{1}$ invariants of the form $\tr{\xb^{-1}}$, 
$\binom{r}{2}$ invariants of the form $\tr{\xb\yb}$, 
$2\binom{r}{2}$ invariants of the form $\tr{\xb\yb^{-1}}$,
$\binom{r}{2}$ invariants of the form $\tr{\xb^{-1}\yb^{-1}}$,
$\binom{r}{2}$ invariants of the form $\tr{\xb\yb\xb^{-1}\yb^{-1}}$,
$2\binom{r}{3}$ invariants of the form $\tr{\xb\yb\zb}$, 
$6\binom{r}{3}$ invariants of the form $\tr{\xb\yb\zb^{-1}}$, 
$3\binom{r}{3}$ invariants of the form $\tr{\xb\yb\zb\yb^{-1}}$, 
$6\binom{r}{3}$ invariants of the form $\tr{\xb\yb^{-1}\zb^{-1}}$,
$6\binom{r}{3}$ invariants of the form $\tr{\xb\yb\zb^{-1}\yb^{-1}}$, 
$\binom{r}{3}$ invariants of the form $\tr{\xb^{-1}\yb^{-1}\zb^{-1}}$,
$5\binom{r}{4}$ invariants of the form $\tr{\wb\xb\yb\zb}$,
$20\binom{r}{4}$ invariants of the form $\tr{\wb\xb\yb\zb^{-1}}$,
$18\binom{r}{4}$ invariants of the form $\tr{\wb\xb\yb^{-1}\zb^{-1}}$, 
$8\binom{r}{4}$ invariants of the form $\tr{\wb\xb\yb\zb\yb^{-1}}$,
$12\binom{r}{5}$ invariants of the form $\tr{\ub\vb\wb\xb\yb}$,
$35\binom{r}{5}$ invariants of the form $\tr{\vb\wb\xb\yb\zb^{-1}}$, and
$15\binom{r}{6}$ invariants of the form $\tr{\ub\vb\wb\xb\yb\zb}$.
\end{thm}

Counting the various generator types we conclude

\begin{cor}
The number of minimal generators for $\C[\SL^{\times r}]^{\SL}$ is $$N_r=\frac{r}{240}\left(396+65r^2-5r^3+19r^4+5r^5\right),$$ and enumerating the 
set of generators from Theorem \ref{maintheorem} by $\{t_1,...,t_{N_r}\}$ we have
$$\mathcal{T}_{N_r}=(t_1,...,t_{N_r}):\X_r\hookrightarrow \C^{N_r}$$ is an affine embedding where $N_r$ is minimal among all affine embeddings 
$\X_r\to \C^{N}$.
\end{cor}

\section{Algebraic Independence}

Using methods from \cite{ATZ} (see also \cite{T3}) we pick out maximal subsets of these minimal generators that are algebraically independent. 

\begin{defn}
Let $k\subset K$ be fields and $B\subset K$ be a set of elements.  Then the set $B$ is algebraically independent over $k$ if for any positive integer 
$n$, any nonzero polynomial $f(x_1,...,x_n)$ with coefficients in $k$, and any set of $n$ distinct elements in $B$, denoted by $b_1,...b_n$, we have $f(b_1,...,b_n)\not=0$.
\end{defn}

Any maximal set of algebraically independent elements has the same cardinality (the \emph{transcendence degree}) and such a maximal set is called 
a \emph{transcendence basis}.  

For an affine variety $X$ over $\C$, the dimension of $X$ (called the {\it Krull dimension}) is equal to the maximal number of independent rational 
functions on $X$; that is, the transcendence degree of the quotient field $\C(X)$ over $\C$.  In this case, it is also equal to the common length of 
all maximal chains of prime ideals in the coordinate ring $\C[X]$, and also to the dimension of a tangent space at a smooth point.  

Suppose the transcendence basis $\{f_1,...,f_d\}$ is in fact in the coordinate ring $\C[X]\subset \C(X)$.  In general a set of funtions whose
differentials give a basis for the cotangent space at a point is called a set of {\it local parameters}.  In \cite{Sh1} it is shown that 
local parameters $\{x_1,...,x_d\}$ exist and that a set of functions $\{f_1,..,f_d\}$ in the local coordinate ring can be written in their terms.  
The Jacobian is the determinant of the $d\times d$ matrix $(\partial f_i/\partial x_j)$.  If the Jacobian is not identically zero on $X$, then at a 
generic smooth point $p$ (themselves generic in $X$) the differentials $d(f_i-f_i(p))$ of these coordinate functions, translated to $p$, give a basis 
for $T^*_p(X)$ and by duality for $T_p(X)$.  Any such set will necessarily be algebraically independent.

Since irreducible representations $\F_r\to \G$ are generic in $\R_r$ (for $r\geq 2$), it follows that generic orbits have dimension equal to 
$\mathrm{dim}\G=m^2-1$.  Thus the Krull dimension of $\X_r$ is $(m^2-1)(r-1)$, as long as $r\geq 2$.  It then follows that at a 
non-singular point $[\rho]$ of $\X_r$ (these are exactly the irreducibles) we wish to obtain $(m^2-1)(r-1)$  
generators of $\C[\X_r]$ whose Jacobian is non-zero on $\X_r$.  This generically provides $(m^2-1)(r-1)$ linearly independent vectors 
in $T_{[\rho]}(\X_r)$ and establishes these invariants are algebraically independent.

For the remainder of this section let $\G=\SLm{3}$.

Our main theorem is 

\begin{thm}\label{indtheorem}
Let $\mathcal{A}=\{\tr{\xb_1},\tr{\xb_2},\tr{\xb_1^{-1}},\tr{\xb_2^{-1}},\tr{\xb_1\xb_2}, \tr{\xb_1^{-1}\xb_2},\\ \tr{\xb_2^{-1}\xb_1},\tr{\xb_1^{-1}\xb_2^{-1}}\}$, and  $\mathcal{B}=\{\tr{\xb_3},\tr{\xb_3^{-1}},\tr{\xb_1\xb_3},\tr{\xb_2\xb_3},\\ \tr{\xb_1\xb_3^{-1}},...,\tr{\xb_r},\tr{\xb_r^{-1}},\tr{\xb_1\xb_r},\tr{\xb_2\xb_r},\tr{\xb_1\xb_r^{-1}}\}$.
Then $\mathcal{A}\cup\mathcal{B}\cup\mathcal{C}$ are local parameters from the minimal generators from Theorem \ref{maintheorem}, where $\mathcal{C}$ is any one of:
\begin{align*}
&\{\tr{\xb_1^{-1}\xb_3^{-1}},\tr{\xb_1^{-1}\xb_3},\tr{\xb_2^{-1}\xb_3^{-1}},...,\tr{\xb_1^{-1}\xb_r^{-1}},\tr{\xb_1^{-1}\xb_r}, \tr{\xb_2^{-1}\xb_r^{-1}} \}\\
& \{ \tr{\xb_1^{-1}\xb_3^{-1}}, \tr{\xb_1^{-1}\xb_3}, \tr{\xb_2^{-1}\xb_3},..., \tr{\xb_1^{-1}\xb_r^{-1}}, \tr{\xb_1^{-1}\xb_r}, \tr{\xb_2^{-1}\xb_r} \}\\
& \{ \tr{\xb_1^{-1}\xb_3^{-1}}, \tr{\xb_1^{-1}\xb_3}, \tr{\xb_2\xb_3^{-1}},...,
\tr{\xb_1^{-1}\xb_r^{-1}}, \tr{\xb_1^{-1}\xb_r}, \tr{\xb_2\xb_r^{-1}} \}\\
& \{ \tr{\xb_1^{-1}\xb_3},\tr{\xb_2^{-1}\xb_3}, \tr{\xb_2\xb_3^{-1}},..., 
\tr{\xb_1^{-1}\xb_r},\tr{\xb_2^{-1}\xb_r}, \tr{\xb_2\xb_r^{-1}} \}\\
&\{ \tr{\xb_1^{-1}\xb_3^{-1}}, \tr{\xb_2\xb_3^{-1}}, \tr{\xb_2^{-1}\xb_3^{-1}},...,
\tr{\xb_1^{-1}\xb_r^{-1}}, \tr{\xb_2\xb_r^{-1}}, \tr{\xb_2^{-1}\xb_r^{-1}} \} .
\end{align*}
In all cases these sets number the Krull dimension of $\X_r$ which is $8r-8$. Consequently, they are maximally algebraically independent in $\C[\X_r]$.
\end{thm}

\begin{proof}

The outline of the proof is as follows.  We will proceed by induction.  For $r=1$ the number of minimal generators equals the dimension of $\X_1$ so there cannot be any relations at all, and for $r=2$ Theorem \ref{indtheorem} was shown earlier in \cite{L2}.
For $r\geq 3$ we calculate the Jacobian matrix of the $8r-8$ trace functions is the statement of Theorem \ref{indtheorem} in the following $8r-8$ independent variables:
\begin{enumerate} 
\item[(a)] $x_{11}^1,x_{22}^1$ from $\xb_1$
\item[(b)]$x^2_{11}, x^2_{21}, x^2_{13}, x^2_{22}, x^2_{23}, x^2_{33}$ from $\xb_2$
\item[(c)]$x^k_{11}, x^k_{12}, x^k_{13}, x^k_{21}, x^k_{22}, x^k_{23}, x^k_{32}, x^k_{33}$ from $\xb_k$ for $3\leq k\leq r$.
\end{enumerate}
Call the set of matrix elements from the above list $\mathcal{V}$.  We then show the determinant of the Jacobian is generically non-zero.  This will establish independence.

We now proceed with the proof.  First, we justify our choice of matrix elements $\mathcal{V}\subset\{x^k_{ij}\}$ by showing them to be independent.  Generically, we can assume that $\xb_1$ is diagonalizable and conjugate it into diagonal form.  Since its determinant is 1, we may write $x_{33}^1=1/(x_{22}^1x_{11}^1)$.  Therefore, we need only $x^1_{11}$ and $x^1_{22}$ from $\xb_1$.

We may still conjugate by any matrix that preserves this normal form, such as by diagonal matrices.  Doing so we will show we can assume the conjugation orbit of $\xb_2$ is independent of its lower diagonal.  The entry $x^2_{31}$ is generically a function of the other matrix elements from $\xb_2$ since we may solve for it in the expression $\mathrm{det}(\xb_2)=1$.  In any event we can always choose either this entry or an element other than the lower diagonal to solve the determinant.  Now we are free to show that the lower diagonal may be assumed to be parameters that vary only in the conjugation orbit.  In deed, conjugate by
$\mathbf{D}=\left(
\begin{array}{lll}
 s & 0 & 0 \\
 0 & x_{21}^2 & 0 \\
 0 & 0 & \frac{x^2_{21} x^2_{32}}{t}
\end{array}
\right)$.  This matrix has non-zero determinant in the free $\C^*$ parameters $s$ and $t$ as long as the lower diagonal of $\xb_2$ is non-zero (generically the case).  Otherwise, the lower diagonal is fixed or we can slightly change the matrix $\mathbf{D}$ to take advantage of one of the zeros on the lower diagonal (when there is {\it only} one zero).

Then $\mathbf{D}^{-1}\xb_2\mathbf{D}=\left(
\begin{array}{lll}
 x^2_{11} & \frac{x^2_{12} x^2_{21}}{s} & \frac{x^2_{13} x^2_{21} x^2_{32}}{s t} \\
 s & x^2_{22} & \frac{x^2_{23} x^2_{32}}{t} \\
 \frac{x^2_{31} s t}{x^2_{21} x^2_{32}} & t & x^2_{33}
\end{array}
\right)$.  As $s$ and $t$ vary, $\xb_1$ is fixed (recall it is now diagonal) and the orbit of $\xb_2$ is fixed.

Thus performing this change of variables shows that we may assume without loss of generality that $x_{11}^1,x_{22}^1, x^2_{ij}\ (1\leq i\leq j\leq 3)$ are independent variables for $\X_2$.  Now for $k\geq 3$ we can likewise solve $\mathrm{det}(\xb_k)=1$ for $x^k_{31}$ (generically).  We thus conclude that 
$x^k_{ij}$, where $1\leq i,j \leq 3$ and $(i,j) \not= (3,1)$ for $k\geq3$, provide the additional independent parameters for $\X_r$ when $r\geq 3$.  

Heuristically, their independence follows since they total $8r-8=\mathrm{dim}\X_r$ and the conjugation action (generically having orbits of dimension 8) has been entirely accounted for (6 degrees of freedom used on $\xb_1$ and 2 degrees of freedom used on $\xb_2$).  A more precise way to say this is that the order $8r-8$ subset of polynomial indeterminates $\mathcal{V}\subset\{x^k_{ij}\}\subset \C[\R]$ are independent in $\C[\R]$, and so are independent in $\C[\X]=\C[\R]^{\G}$ as long as they (generically) distinguish orbits.  Since we just (constructively) showed that this set does generically distinguish all orbits (their closures to be precise), then we have shown independence.  The set $\mathcal{V}$ is a set of parameters we can now work with.

We now say a word or two about the Jacobian matrix.  The functions $\mathcal{A}\cup\mathcal{B}\cup\mathcal{C}$ are generically functions of the variables $\mathcal{V}$. Letting $f_1,..,f_{8r-8}$ be the trace functions from $\mathcal{A}\cup\mathcal{B}\cup\mathcal{C}$ and $z_1,...,z_{8r-8}$ be the variables $\mathcal{V}$, the Jacobian is the $(8r-8) \times (8r-8)$ matrix of partial derivatives $(\frac{\partial f_i}{\partial z_j})$.  If there is a dependence relation among the $f_i$'s then locally (that is on an open affine subset) the Jacobian will have zero determinant since locally these functions cannot give a full dimensional tangent space.  So computing this determinant and finding it non-zero at a general point in a Zariski open set shows the functions are algebraically independent.

We now proceed with the induction.  The cases $r=1,2$ are done.  Suppose now that $r_0\geq 3$ and that for all $r<r_0$ the Jacobian is non-singular.

Putting the $8$ trace functions which are in terms of $\xb_{r_0}$ in the last $8$ rows and the 8 variables $x_{ij}^{r_0}$ in the last 8 columns we get a block diagonal matrix $\left(
\begin{array}{cc}
 M & 0  \\
 P & N 
\end{array}
\right)$ and so by induction $M$ is non-singular.  The $8(r-2)\times 8$ block of zeros arises since the first $8(r-2)$ trace functions are constant with respect to $\xb_{r_0}$ and the last $8$ columns come from differentiating with respect to the indeterminates $x_{ij}^{r_0}$.  The block form of the matrix implies that its determinant is the product of the determinant of $M$ and that of $N$.  It remains to show that the last eight traces are independent in the variables from the last matrix $\xb_r$; that is, the $8\times 8$ matrix $N$ is non-singular (generically).

For all choices of $\mathcal{C}$, using {\it Mathematica} (code is available upon request), we calculate this subdeterminant and evaluate at random unimodular matrices (those with determinant 1); finding it non-zero.  If there was a relation the determinant would be identically zero and so any non-zero evaluation shows independence. We note that we only need to test the $r=3$ case since all $8\times 8$ lower right blocks $N$ are identical excepting labels.\end{proof}

\begin{rem}
Using the same method as above it is not hard to prove that the natural affine lifts $($replace exponents of $-1$ with exponents of $2)$ of the generators in the main theorem with $\{\tr{\xb_1^3},...,\tr{\xb_r^3}\}$ added to the sets form a maximal algebraically independent set for $\mathfrak{gl}(3,\C)^{\times r}\aq \G$.  For instance,
\begin{align*}&\{\tr{\xb_1},\tr{\xb_2},\tr{\xb_1^{2}},\tr{\xb_2^{2}},\tr{\xb_1^{3}},\tr{\xb_2^{3}},\tr{\xb_1\xb_2},\tr{\xb_1^{2}\xb_2},\tr{\xb_2^{2}\xb_1},\\ &\tr{\xb_1^{2}\xb_2^{2}}, \tr{\xb_3},\tr{\xb_3^{2}},\tr{\xb_3^{3}},\tr{\xb_1\xb_3},\tr{\xb_2\xb_3},
\tr{\xb_1\xb_3^{2}},\tr{\xb_1^{2}\xb_3^{2}},\\ &\tr{\xb_1^{2}\xb_3}, \tr{\xb_2^{2}\xb_3^{2}},...,\tr{\xb_r},\tr{\xb_r^{2}},\tr{\xb_r^{3}},\tr{\xb_1\xb_r},\tr{\xb_2\xb_r},\tr{\xb_1\xb_r^{2}},\\ &\tr{\xb_1^{2}\xb_r^{2}},\tr{\xb_1^{2}\xb_r}, \tr{\xb_2^{2}\xb_r^{2}}\}, \text{ are independent and maximal.} \end{align*}
\end{rem}

\section{Pulling Back Parameters and the Magnus Trace Map}
In this section, let $\mathfrak{Y}_r=\glm{m}^{r}\aq \SLm{m}$, $\X_r=\SLm{m}^r\aq \SLm{m}$, and $$\Pi_{\G}:\C[\mathfrak{Y_r}]\to \C[\X_r]\approx \C[\mathfrak{Y}_r]/(\mathrm{det}(\xb_1)-1,...,\mathrm{det}(\xb_r)-1)$$ be the projection discussed in the introduction.

A set of generators for $\C[\Y_r]$ or $\C[\X_r]$ of the form $\{\tr{A_{i_1}A_{i_2}\cdots A_{i_j}}\}$ are called {\it Procesi generators}, if additionally no generator has the form $\tr{\wb_1\xb^m\wb_2}$ where at least one of the words $\wb_i$ is not the identity (for $\C[\X_r]$ we additionally require this if both $\wb_i=\id$).  Using the characteristic polynomial $\sum c_k(\xb)\xb^{n-k}=0$ one can always arrange for any set of minimal generators of $\C[\Y_r]$ or $\C[\X_r]$ to be Procesi generators.  We will call a maximal set of algebraically independent Procesi generators {\it Procesi parameters}.  Such a maximal set always has order equal to the Krull dimension of $\Y_r$ or $\X_r$, respectively.

In \cite{L4} we show that minimal Procesi generators of $\C[\glm{m}^{\times r}\aq \SLm{m}]$ project by $\Pi_{\G}$ to minimal generators of $\C[\SLm{m}^{\times r}\aq \SLm{m}]$.  One may show that a set of algebraically independent Procesi generators of $\SLm{2}^{\times r}\aq \SLm{2}$ pull back by $\Pi_{\SLm{2}}$ to algebraically independent generators of $\glm{2}^{\times r}\aq \SLm{2}$, and the above work and remark from the last section shows the same is true for $\SLm{3}$.

We now prove 

\begin{thm}Let $B\subset \C[\X_r]$ be a set of Procesi parameters for $\X_r$.  Then there is a collection of elements $\overline{B}\subset \Pi_{\G}^{-1}(B)$ so that $\{\tr{\xb_1^m},...,\tr{\xb_r^m}\}\cup \overline{B}\subset \C[\Y_r]$ is a set of Procesi parameters for $\Y_r$.  Said shortly, $\pi$ pulls back Procesi parameters.
\end{thm}
\begin{proof}
By assumption the parameters for $\X_r$ are given by traces of words in generic matrices where all letters have positive exponents; that is in the form $\tr{\xb_{i_1}\xb_{i_2}\cdots \xb_{i_l}}$,  and the collection $B$ does not include any parameter of the form $\tr{\xb^m}$.  Then these invariants extend to invariants of $\mathfrak{Y}_r$; that is, for such generators, $\Pi_{\G}(\tr{\wb})=\tr{\wb}{\big|_{\X_r}}$.  If there was a dependence relation among these extensions to all of $\mathfrak{Y}_r$, then restricting would provide a dependence relation on all of $\X_r$, which cannot exist by assumption.  Thus $B$ pulls back to an independent set $\overline{B}$.

Let $\C[\overline{B}]$ be the subring of $\C[\Y_r]$ generated by $\overline{B}$.  The collection $\{\tr{\xb_k^{m-j}}\}$ for $1\leq j\leq m-1$ is included in $\overline{B}$, since they must be included in any set of Procesi parameters for $\X_r$; the latter following since the Procesi parameters for $\X_1$ are unique, and setting $\xb_2=\cdots=\xb_r=\id$ recovers $\C[\X_1]$ as a subring of $\C[\X_r]$.  Note that they are independent since $\glm{m}\aq \SLm{m}=\C^m$. Suppose $\C[\overline{B}][\tr{\xb^m_1},...,\tr{\xb^m_r}]\subset \C[\mathfrak{Y}_r]$ is an algebraic extension of $\C[\overline{B}]$.  Then locally, using the characteristic polynomial, we can express the determinant in terms of the other coefficients of the characteristic polynomial, which we know to be impossible generally, since the coefficients of the characteristic polynomial are already independent.  For example, by letting all elements $\xb_i$ be the identity except for one, say $\xb_{i_0}$, we would be able to locally ascertain $\mathrm{det}(\xb_{i_0})$ from $\tr{\xb_{i_0}^{m-k}}$ for $1\leq k\leq m-1$; which we cannot.  Thus, it is a transcendental extention.  Therefore, $\overline{B}\cup \{\tr{\xb^m_1},...,\tr{\xb^m_r}\}$ are algebraically independent and since they number $(m^2-1)(r-1)+r=(r-1)m^2+1=\mathrm{dim}\mathfrak{Y}_r$, they are a set of parameters for $\mathfrak{Y}_r$.
\end{proof}

We also take a moment to contrast the fact in \cite{L4} that $N_r$ minimal generators provide an affine embedding $\X_r\hookrightarrow\C^{N_r}$, where $N_r$ is minimal among all such embeddings.

Given a set of independent generators $\mathcal{I}$ of order $|\mathcal{I}|$ one can construct a morphism $\mathsf{MagTr}_{\mathcal{I}}:\X_r\to \C^{|\mathcal{I}|}$, called the {\it Magnus Trace Map} in \cite{Fl}.  This map factors through the affine embedding $\X_r\to\C^{N_r}$ with a projection when $\mathcal{I}$ is a subset (as it can always be taken to be) of a minimal set of generators defining the embedding.

It is not hard to prove

\begin{prop}
$\mathsf{MagTr}_{\mathcal{I}}$ is dominant and generically submersive.
\end{prop}
\begin{proof}
Since the algebraically independent generators induce an injection $\C[B]\hookrightarrow \C[\X_r]$ we automatically get 
dominance (\cite{Ha}, page 81), and then this immediately implies there exists an open dense set $U\subset \X_r$ where
 $\mathsf{MagTr}_{\mathcal{I}}$ is a smooth submersion (\cite{Ha}, page 271).
\end{proof}

Suppose $\mathcal{I}\subset \{\tr{\xb_{i_1}\xb_{i_2}\cdots \xb_{i_l}}\}$ is taken to be a maximal.  For $\SLm{m}$ and $r=1$ it is surjective since it defines an isomorphism.  Surprisingly, we also have an isomorphism for $\SLm{2}$ and $r=2$ (see \cite{G9} for more about the Fricke-Klein-Vogt Theorem).  For $\SLm{3}$ and $r=2$ and for $\SLm{2}$ and $r=3$ it is a branched double cover.  It is natural to conjecture that the map is surjective in general.  However, for $\G=\SLm{2}$, Florentino \cite{Fl} recently showed that it is not surjective for $r\geq 4$.

We conjecture the same is true for $\SLm{m}$ and $r\geq 2$ for $m\geq 3$ excepting only $(m,r)=(3,2)$.  In other words, we now make 

\begin{conj}
Let $\mathcal{I}\subset\{\tr{\xb_{i_1}\xb_{i_2}\cdots \xb_{i_l}}\}$ be a set of parameters for $\X_r$.  Then $\mathsf{MagTr}_{\mathcal{I}}$ is only surjective for the cases $(m,r)=$ $(1,r),$ $(m,1),$ $(2,2),$ $(2,3),$ $(3,2).$
\end{conj}

\begin{rem}
Additionally, in \cite{Fl} it is shown that a representation $\rho:\pi\to \SLm{2}$ where $\pi$ is a finitely generated group with 
generators $\{g_1,...,g_r\}$ is reducible if and only if all triples $(\rho(g_{i_1}),\rho(g_{i_2}),\rho(g_{i_3}))$ are simultaneously reducible by 
conjugation.  If we replace $\SLm{2}$ by $\SLm{m}$ and $i_3$ by $i_{d(m)}$ where $d(m)$ is the degree of nilpotency discussed in the introduction, 
then Florentino's statement remains valid.  This follows since globally all invariants are $\tr{\xb_{i_1}\xb_{i_2}\cdots \xb_{i_{d(m)}}}$, and the 
points of $\X_r$ are distinguished by such traces.  It is important to observe that it does not matter that the points in $\X_r$ correspond to 
unions of orbits, since such a union contains either one and only one orbit of an irreducible representation or it contain only reducible 
representations (in different orbits, but all reducible).
\end{rem}

\section{Closing Remarks}
In this section we take the opportunity to briefly describe our general outlook on the project that this paper, in part, contributes.

In \cite{L2} we showed the following theorem.

\begin{thm} Let $\X=\SL^{\times 2} \aq \SL$.  Then the following hold:
\begin{enumerate}
\item[$(i)$] $\C[\X]$ is minimally generated by the nine affine coordinate functions
\begin{align*}
\mathcal{G}=&\{\tr{\xb_1},\tr{\xb_2},\tr{\xb_1\xb_2},\tr{\xb_1^{-1}},\tr{\xb_2^{-1}},\tr{\xb_1\xb_2^{-1}},\\ &\tr{\xb_2\xb_1^{-1}},\tr{\xb_1^{-1}\xb_2^{-1}},\tr{\xb_1\xb_2\xb_1^{-1}\xb_2^{-1}}\}.
\end{align*}
\item[$(ii)$] The eight elements in $\mathcal{G}\backslash \{\tr{\xb_1\xb_2\xb_1^{-1}\xb_2^{-1}}\}$ are a maximal algebraically independent subset and
are local parameters.
\item[$(iii)$] $\tr{\xb_1\xb_2\xb_1^{-1}\xb_2^{-1}}$ satisfies a monic (degree 2) relation over the algebraically independent generators.  It generates the ideal.
\item[$(iv)$] $\mathrm{Out}(\F_2)$ acts on $\C[\X]$ and has an order $8$ subgroup which acts as a permutation group on the independent generators; as such distinguishes them.
\end{enumerate}
\end{thm}

In \cite{L4} we generalize part $(i)$ of this theorem to $\X_r$ by describing minimal generators for any value of $r$.  In this paper we generalize 
part $(ii)$ of this theorem again to $\X_r$ for any value of $r$.  We are currently exploring a 
generalization of part $(iv)$ to $\X_r$.  Generalizing part $(iii)$ seems to be a very hard problem.  

Recently, exciting new results using methods 
similar to those in \cite{AP} were established in \cite{BD1} concerning the ideal of relations for generic $3\times 3$ matrices.  In 
particular, the minimal degree of generators of the ideal of relations was found to be $7$ and the degree $7$ relations were then classified in 
general.  Using this we can get like results for $\X_r$, but this is just the beginning.  

It would seem that either new ideas are needed to 
solve the relations problem in general, or a massive computational project.  On the other hand, to solve the problem in general, one would like to 
know the minimal generators to work with (part $(i)$ above and its generalization), the subsets of generators that do not admit relations (part $(ii)$ 
above and its generalization), and group actions which can simplify the form of relations (part $(iv)$ above and its generalization).

\subsection*{Acknowledgments}
The author thanks his advisor Bill Goldman for introducing him to this topic and for his support and encouragement.  This work began at Kansas 
State University during the 2006-2007 academic year, and was completed at Instituto Superior T\'ecnico during the 2007-2008 academic year.  
The author thanks KSU and IST for financial support and for providing a stimulating environments to work.  We thank Mara Neusel, Frank Grosshans, 
Zongzhu Lin, and Gustavo Granja for general interest in this work.  We extend a special thanks to Carlos Florentino for interesting conversations 
and for motivating parts of Section 3.


\begin{thebibliography}{MFK94}

\bibitem[AP89]{AP}
Silvana Abeasis and Marilena Pittaluga, \emph{On a minimal set of generators
  for the invariants of {$3\times 3$} matrices}, Comm. Algebra \textbf{17}
  (1989), no.~2, 487--499. \MR{MR978487 (90d:15021)}

\bibitem[ATZ94]{ATZ}
Helmer Aslaksen, Eng-Chye Tan, and Chen-bo Zhu, \emph{Generators and relations
  of invariants of {$2\times 2$} matrices}, Comm. Algebra \textbf{22} (1994),
  no.~5, 1821--1832. \MR{MR1264744 (95c:15061)}

\bibitem[BD07]{BD1}
F.~Benanti and V.~Drensky, \emph{Defining relations of minimal degree of the
  trace algebra of {$3\times 3$} matrices}, C. R. Acad. Bulgare Sci.
  \textbf{60} (2007), no.~2, 103--110. \MR{MR2301687}

\bibitem[Dol03]{Do}
Igor Dolgachev, \emph{Lectures on invariant theory}, London Mathematical
  Society Lecture Note Series, vol. 296, Cambridge University Press, Cambridge,
  2003. \MR{MR2004511 (2004g:14051)}

\bibitem[For02]{F}
Edward Formanek, \emph{The ring of generic matrices}, J. Algebra, \textbf{258} (2002), no.~1,310--320. \MR{MR1958908 (2004a:16039)}

\bibitem[Flo06]{Fl}
Carlos A.~A. Florentino, \emph{Invariants of {$2\times2$} matrices, irreducible
  {${\rm SL}(2,{\mathbb{C}})$} characters and the {M}agnus trace map}, Geom.
  Dedicata \textbf{121} (2006), 167--186. \MR{MR2276242 (2007k:14093)}

\bibitem[Gol08]{G9}
William~M. Goldman, \emph{Trace coordinates on fricke spaces of some simple
  hyperbolic surfaces}, EMS Publishing House, Z\"urich, 2008, Handbook of
  Teichm\"uller theory II ( A. Papadopoulos, editor).

\bibitem[Har77]{Ha}
Robin Hartshorne, \emph{Algebraic geometry}, Springer-Verlag, New York, 1977,
  Graduate Texts in Mathematics, No. 52. \MR{MR0463157 (57 \#3116)}

\bibitem[Kir67]{Ki}
A. A.~ Kirillov, \emph{On certain algebras with division over the field of rational functions}, Funkcional. Anal. i Priložen \textbf{1} (1967), 
101--102. \MR{MR0207751 (34 \#7566)}

\bibitem[Law07]{L2}
Sean Lawton, \emph{Generators, relations and symmetries in pairs of $3\times 3$
  unimodular matrices}, J. Algebra \textbf{313} (2007), no.~2, 882--801.
\MR{MR2329569 (2008k:16039)}

\bibitem[Law08]{L4}
\bysame, \emph{Minimal affine coordinates for $\mathrm{SL}(3,\mathbb{C})$
  character varieties of free groups}, J. Algebra \textbf{320} (2008), no.~10, 3773--3810.
\MR{{MR2457722 (2009j:20060)}}


\bibitem[Lun75]{Lu2}
D.~Luna, \emph{Sur certaines op\'erations diff\'erentiables des groupes de
  {L}ie}, Amer. J. Math. \textbf{97} (1975), 172--181. \MR{MR0364272 (51
  \#527)}

\bibitem[Lun76]{Lu3}
Domingo Luna, \emph{Fonctions diff\'erentiables invariantes sous l'op\'eration
  d'un groupe r\'eductif}, Ann. Inst. Fourier (Grenoble) \textbf{26} (1976),
  no.~1, ix, 33--49. \MR{MR0423398 (54 \#11377)}

\bibitem[MFK94]{MFK}
D.~Mumford, J.~Fogarty, and F.~Kirwan, \emph{Geometric invariant theory}, third
  ed., Ergebnisse der Mathematik und ihrer Grenzgebiete (2) [Results in
  Mathematics and Related Areas (2)], vol.~34, Springer-Verlag, Berlin, 1994.
  \MR{MR1304906 (95m:14012)}

\bibitem[Nag64]{Na}
Masayoshi Nagata, \emph{Invariants of a group in an affine ring}, J. Math.
  Kyoto Univ. \textbf{3} (1963/1964), 369--377. \MR{MR0179268 (31 \#3516)}

\bibitem[Pro76]{P1}
C.~Procesi, \emph{The invariant theory of {$n\times n$} matrices}, Advances in
  Math. \textbf{19} (1976), no.~3, 306--381. \MR{MR0419491 (54 \#7512)}

\bibitem[Raz74]{Ra}
J. ~Razmyslov, \emph{Identities with trace in full matrix algebras over a field of characteristic zero}, Izv. Akad. Nauk SSSR Ser. Mat. \textbf{38} 
(1974), 723--756. \MR{MR0506414 (58 \#22158)}

\bibitem[Sch01]{Sch2}
Gerald~W. Schwarz, \emph{Algebraic quotients of compact group actions}, J.
  Algebra \textbf{244} (2001), no.~2, 365--378. \MR{MR1857750 (2002h:14108)}

\bibitem[Sha94]{Sh1}
Igor~R. Shafarevich, \emph{Basic algebraic geometry. 1}, second ed.,
  Springer-Verlag, Berlin, 1994, Varieties in projective space, Translated from
  the 1988 Russian edition and with notes by Miles Reid. \MR{MR1328833
  (95m:14001)}

\bibitem[Ter88]{T3}
Yasuo Teranishi, \emph{The {H}ilbert series of rings of matrix concomitants},
  Nagoya Math. J. \textbf{111} (1988), 143--156. \MR{MR961222 (90a:16017)}

\end{thebibliography}


\def\cdprime{$''$} \def\cdprime{$''$} \def\cprime{$'$} \def\cprime{$'$}
  \def\cprime{$'$} \def\cprime{$'$}
\providecommand{\bysame}{\leavevmode\hbox to3em{\hrulefill}\thinspace}
\providecommand{\MR}{\relax\ifhmode\unskip\space\fi MR }
\providecommand{\MRhref}[2]{%
  \href{http://www.ams.org/mathscinet-getitem?mr=#1}{#2}
}
\providecommand{\href}[2]{#2}

\end{document}